\documentclass[a4paper,final]{amsart}

\usepackage[utf8]{inputenc}
\usepackage[UKenglish]{babel}
\usepackage{mathtools}

\usepackage{amsmath, amsthm, amsfonts, amssymb, color}
\usepackage{mathrsfs}
\usepackage{dsfont}
\usepackage{enumerate}
\usepackage{natbib}
\usepackage[margin=1.5in]{geometry}

\theoremstyle{plain}
\newtheorem{theorem}{Theorem}[section]
\newtheorem{corollary}[theorem]{Corollary}
\newtheorem{lemma}[theorem]{Lemma}
\newtheorem{prop}[theorem]{Proposition}

\theoremstyle{definition}
\newtheorem{remark}[theorem]{Remark}

\newcommand{\real}{\mathds{R}}
\newcommand{\Pp}{\mathds{P}}
\newcommand{\Ee}{\mathds{E}}
\newcommand{\I}{\mathds{1}}
\newcommand{\Dcal}{\mathcal{D}}
\newcommand{\Fcal}{\mathcal{F}}
\newcommand{\Lcal}{\mathcal{L}}
\newcommand{\sgn}{\operatorname{sgn}}

\begin{document}

\title[Asymptotics of time changed Wiener processes]{Asymptotic behaviour and functional limit theorems for a time changed Wiener process}

\author[Y.\ Kondratiev]{Yuri Kondratiev}
\address{Fakult\"at Mathematik, Bielefeld Universit\"at,
33615 Bielefeld, Germany}
\email{kondrat@math.uni-bielefeld.de}

\author[Y.\ Mishura]{Yuliya Mishura}
\address{Department of Probability, Statistics and Actuarial Mathematics,\\ Taras Shevchenko National University of Kyiv, 01601  Kyiv, Ukraine} \email{myus@univ.kiev.ua}

\author[R.L.\ Schilling]{Ren\'e L.\ Schilling}
\address{Institut f\"{u}r Mathematische Stochastik, Fakult\"{a}t Mathematik, TU Dresden, 01062 Dresden, Germany}
\email{rene.schilling@tu-dresden.de}

\begin{abstract}
   We study the asymptotic behaviour of a properly normalized time changed Wiener processes. The time change reflects the fact that we consider the Laplace operator (which generates a Wiener process) multiplied by a possibly degenerate state-space dependent intensity $\lambda(x)$. Applying a functional limit theorem for the superposition of stochastic processes, we prove functional limit theorems for the normalized time changed Wiener process. The normalization depends on the asymptotic behaviour of the intensity function $\lambda$. One of the possible limits is a skew Brownian motion.
\end{abstract}


\maketitle

\section{Introduction}\label{intro}
In this note we study the asymptotic behaviour of normalized, time changed Wiener processes.  The time changes are (generalized) inverses of additive functionals of the Wiener process. Our motivation for this kind of problem comes from the study of solutions to parabolic Cauchy problems of the following type
\begin{gather}\begin{alignedat}{3}\label{cauchy}
    \frac{\partial}{\partial t} u(t,x) &= \lambda(x)\Delta u(t,x), &\quad& t\ge 0,\; x\in \real^d,\\
                                u(0,x) &= f(x), &\quad& t=0,\; x\in \real^d,
\end{alignedat}\end{gather}
where the coefficient $\lambda$ can be irregular and even degenerate. In many situations, it is assumed that the diffusion coefficient is uniformly elliptic -- here: $\lambda(x)$ is bounded away from zero; degenerate problems are often treated with the method of vanishing viscosity, see e.g.\ \cite{ole-rad} (\S III.2) and \cite{bog-et-al} (\S 6.7(ii)) and the literature mentioned there, which can be quite cumbersome.

An alternative method to study such problem is to use probabilistic approach. Assume that we can construct a Markov process $X$, starting from every point $x\in\real^d$, and with infinitesimal generator $\lambda(x)\Delta$. In this situation, we have the following stochastic representation of the parabolic Cauchy problem \eqref{cauchy}
\begin{gather*}
    u(t,x)= \Ee^x [f(X_t)].
\end{gather*}
Yet, a direct construction of this process may be difficult (if possible at all) if one wants to use stochastic differential equations (SDEs) if the diffusion coefficient $\sqrt{\lambda(x)}$ is irregular and degenerate. In particular, in the study of diffusions in random media
random coefficients $\lambda =\lambda(\omega,x)$ appear, which have an additive form with respect to the points in the medium. This structure corresponds to the energy of diffusing particle in the medium.   In general, such coefficients will   not be uniformly positive.

A possible way out is offered by the theory of Dirichlet forms, e.g.\ \cite{Chen}, but the resulting process may only be defined up to an exceptional set. We will use a different approach, namely additive functionals and random time changes. Recall the following statement which is proved differently and for different purposes in \cite{BSW} and \cite{kurtz}.
\begin{prop}\label{prop1}
    Let $\left\{X_t,\, t\geq 0\right\}$ be a $d$-dimensional Feller process with natural filtration $\Fcal^X$, $\Fcal_t^X = \sigma(X_s, s\leq t)$, and generator $(A, \Dcal(A))$. Assume that the function $\lambda\in C_b(\real^d)$ is real-valued and strictly positive. Denote by
    \begin{gather*}
        S_X(t,\omega)=\int_0^t\frac{ds}{\lambda(X_s(\omega))}
    \end{gather*}
    and by $\tau_t(\omega) = \inf\{u>0: S_X(u,\omega)>t\}$ $(\inf\emptyset:=+\infty)$ the generalized right-continuous inverse. Then the time changed process
    $
        \left\{X_{\tau_t},\, t\geq 0\right\}
    $
    with filtration
    $
        \left\{\Fcal_{\tau_t}^X,\,t\geq 0\right\}
    $
    is again a Feller process and its generator is the closure of $(\lambda(\cdot)A, \Dcal(A))$.
\end{prop}

Due to the different recurrence and transience behaviour of one- and multi-dimensional diffusion processes, we restrict ourselves to the case $d=1$; for simplicity, we consider the case where $X_t=B_t$ is a standard Wiener process. Throughout, we assume that $\lambda:\real\to [0,\infty)$ is a positive and measurable function which is Lebesgue a.e.\ strictly positive, i.e.\ $\mathrm{Leb}\{x:\lambda(x)=0\}=0$. In our main theorem, we do not require that $\lambda$ is bounded or bounded away from zero. Since $B_t$ has a transition density w.r.t.\ Lebesgue measure, the a.e.\ strict positivity of $\lambda$ guarantees that $S_B(t,\omega)$ is a.s.\ strictly increasing and continuous; thus, $\tau_t$ is also strictly increasing and continuous. Finally, we assume that $1/\lambda$ is locally integrable but not necessarily of class $L^1(\real,dx)$. Note that under these assumptions, the time changed process $B_{\tau_t}$ is a diffusion process and a martingale w.r.t.\ the filtration $\Fcal_{\tau_t}$; in fact, writing $Y_t = B_{\tau_t}$, we have
\begin{gather}\label{mart-repres}
    Y_t=Y_0+\int_0^t \sqrt{\lambda(Y_s)}\,d\tilde{B}_s,
\end{gather}
with some Wiener process $\tilde{B}$.

Our aim is to establish functional limit theorems for the process $B_{\tau_t}$ without resorting to the martingale representation \eqref{mart-repres} or the Feller property and the form of its generator; instead, we want to focus on the pathwise representation as superposition of two stochastic processes. Using this approach, we do not need that the time change $\tau$ and the original process $X$ are independent (e.g.\ as in Bochner's subordination) nor regularity of $\lambda$ (as in the Feller case) or non-degeneracy (as for the SDE approach). There exists a substantial literature on random-time changes, using various perspectives. Let us mention but a few of them: The monograph \cite{Chen} treats time changes of symmetric Markov processes, and \cite{harl} describes  the representation  of semi-Markov processes  in the form of a Markov process, transformed by a time change. The paper \cite{Magdziarz} is  discusses the asymptotic behaviour of the standard one-dimensional Brownian motion time changed by the (generalized) inverse of an (independent) subordinator. A detailed bibliography on the properties of processes after   time change  is contained in the paper  \cite{amp}.

Let us briefly outline the present note. Section \ref{sec-1} contains the main results on functional limit theorems for the time changed Wiener process. In Section \ref{sub1} we briefly recall (a modification of) a theorem on weak convergence of superposition of stochastic processes by \cite{Silvestrov}. This is then applied to the situation at hand, beginning with relatively strong assumptions on $\lambda$ (Section \ref{sub2}, existence of the limits $\lambda(\pm\infty)$) which are replaced in Section~\ref{sub3} by integral conditions which allow us to consider, e.g.\ polynomially growing or periodic $\lambda$'s. The latter case is related to the theory of homogenization (see, e.g.\ \cite{arm} and the references given there). The appendix contains, for the readers' convenience, some auxiliary results from the literature.

\section{Functional limit theorems for a time changed Wiener process}\label{sec-1}

Consider a one-dimensional Wiener process  $B=\{ B_t, \, t\geq0 \}$, starting from $0$, and denote by $\Fcal^B = \{\Fcal^{B}_{t}, \, t\geq0 \}$ be its natural filtration. Throughout, $\lambda:\real\to [0,\infty)$ is a  measurable function such that $f:\real\to (0,\infty]$, $f(x) := \lambda(x)^{-1}$ is locally integrable. Under these assumptions, the following additive functional of Brownian motion $S_B(x) = \int_{0}^{x} \lambda(B_s)^{-1}\,ds$ is well-defined. Denote by $\Lcal_t^a(B)$ the local time of the Wiener process $B$ at the point $a$ and up to time $t$. Since $B$ is recurrent, $\Lcal_\infty^a(B)=\infty$ a.s., and so
\begin{gather}\label{local}
    S_B(\infty)
    =\int_{0}^{\infty} \frac{ds}{\lambda(B_s)}
    = \int_{\real} \Lcal_\infty^a(B) \frac{da}{\lambda (a)}
    = \infty
    \quad \text{a.s.}
\end{gather}
By $\tau_t  = \inf \{ x>0: \, S_B(x) > t \}$ we denote the generalized right-inverse of the additive functional $S_B$. As usual, $\inf\emptyset:=\infty$; note that $\tau_0 = 0$. Since $S_B(x)$ is continuous and strictly increasing, we have
\begin{align*}
    \{ \tau_t \leq u \} &= \left \{S_B(u)=\int_{0}^{u} \frac{ds}{\lambda(B_s)} \geq t \right \} \in \Fcal_{u}^{B}.
\end{align*}
Therefore, $\tau_t$ is a stopping time w.r.t.\ the filtration $\Fcal^B$; moreover,  $\tau_t < \infty$ a.s., and we can consider the time changed process $B_{\tau_t}$.  We are going to apply to this process limit theorems for the superposition of stochastic  processes. Let us point out that, in general, $B$ and $\tau$ are not independent nor is $B_{\tau_t}$ a   square-integrable martingale.   In view of Wald's identities, a necessary and sufficient condition for the latter is $\Ee\left[\tau_t\right]<\infty$. The following lemma gives a typical sufficient condition for $\Ee\left[\tau_t\right]<\infty$.

\begin{lemma}
    If $\lambda$ has a unform lower bound $ \lambda(x) \geq c >0$  for all $x \in \real$, then $\Ee\left[\tau_t\right]<\infty$ for all $t>0$.
\end{lemma}
\begin{proof}
The claim follows easily from the following identity
\begin{align*}
    \Ee\left[\tau_t\right]
    = \int_{0}^{\infty} \Pp \left \{ \tau_t \geq x \right \} dx
    &= \int_{0}^{\infty} \Pp \left \{ \int_{0}^{x} \frac{ds}{\lambda(B_s)}  \leq t \right \} dx \\
    &= \int_{0}^{\frac{t}{c}} \Pp \left \{ \int_{0}^{x} \frac{ds}{\lambda(B_s)}   \leq t\right \} dx \leq \frac{t}{c}.
\qedhere
\end{align*}
\end{proof}

Since $t\mapsto\tau_t$ is continuous, $t\mapsto B_{\tau_t}$ and $t\mapsto B_{\tau_{nt}}$ are continuous, but we have not much further information. This means, in particular, that we have to use methods which do not rely on ($L^2$-)martingale methods or independence of the time change. The aim of the present paper is to identify the normalizing factor $\phi(n)\to\infty$ as $n\to\infty$ such that, for a suitable limiting process $\zeta=\{\zeta_t, t\ge 0\}$
\begin{gather*}
    \left\{\frac{B_{\tau_{nt}}}{\phi(n)},\, t\in[0,T]\right\} \xRightarrow[]{\phantom{\mathrm{d}}} \left\{\zeta_t, t\in[0,T]\right\}
    \quad\text{as\ } n \to \infty
\end{gather*}
in the sense of weak convergence of stochastic processes on any interval $[0,T]$, $T>0$. As usual, weak convergence is denoted by the symbol `$\xRightarrow[]{\phantom{\mathrm{d}}}$'.

\subsection{Functional limit theorem for the superposition of stochastic processes}\label{sub1}

We will use the following modification of theorem on weak convergence of superposition of stochastic processes, see Theorem 3.2.1 and its generalization in Section~3.2 of the book \cite{Silvestrov}. By `$\xRightarrow[]{\mathrm{d}}$' we denote the weak convergence of the finite-dimensional distributions, whereas `$\xRightarrow[]{\phantom{\mathrm{d}}}$' means weak convergence of probability measures.

\begin{theorem}\label{th-cond-1}
    Let $\left\{(Z_{n}(t), \nu_{n}(t)), t \geq 0\right\}_{n\geq 0}$ be a sequence of continuous stochastic processes such that $\nu_n(t)\geq 0$ for all $n\geq 1$, $t\geq 0$. Assume that the following conditions hold for all $T>0$
    \begin{enumerate}[\upshape i)]
    \item \label{th-cond-1i}
        There is some dense subset $U\subset [0,T]$ such that $0\in U$ and
        $\{(\nu_{n}(t) , Z_{n}(t)),\, t\in U\}  \xRightarrow[]{\mathrm{d}} \{(\nu_{0}(t) , Z_ {0}(t)), \,  t  \in U\}$ as $n \to \infty$.
    \item\label{th-cond-1ii}
        For every $\epsilon>0,\quad \displaystyle\lim_{h \downarrow 0} \limsup_{n\to \infty} \Pp \left\{\sup_{0 \leq t_1 \leq t_2 \leq (t_1+h) \wedge T} |Z_n (t_2) - Z_n (t_1)|> \epsilon \right\}=0$.
    \end{enumerate}
    Then, on every interval $[0,T]$, the superposition $\zeta_n (\cdot) := Z_n (\nu_n (\cdot))$ converges weakly:
    \begin{gather*}
        \zeta_n (\cdot) \xRightarrow[]{\phantom{\mathrm{d}}} \zeta_0 (\cdot) = Z_0 (\nu_0 (\cdot)).
    \end{gather*}
 \end{theorem}

\subsection{Functional limit theorems for the case of converging intensity $\lambda$}\label{sub2}

We are going to apply Theorem~\ref{th-cond-1} with $\phi(n) := \sqrt{n}$, $Z_n(t) := \frac{1}{\sqrt{n}}B_{nt}$ and $\nu_n(t) :=\frac{1}{{n}}\tau_{nt}$. In order to demonstrate the essence of our method, we want to keep things simple for the moment and we assume that the function $\lambda(s)$ has limits at $\pm \infty$ and that it is uniformly bounded away from zero. These conditions will then be relaxed in a second step.

\begin{theorem}\label{lim-thm-1}
    Let $\lambda:\real\to (0,\infty)$ be measurable such that $1/\lambda$ is locally integrable and there exist two numbers, $a_\pm >0$ such that
    \begin{gather*}
        \lim_{x\to\pm\infty}\lambda(x) = a_\pm.
    \end{gather*}
    If $ \lambda(x)\ge c>0$ for all $x \in \real$, then on any interval $[0,T]$
    \begin{gather*}
        \left\{\frac{B_{\tau_{nt}}}{\sqrt{n}},\,t\in[0,T]\right\}
        \xRightarrow[]{\phantom{\mathrm{d}}}
        \left\{W(\eta^{-1} (t)),\, t\in[0,T]\right\},
    \end{gather*}
    where $W$ is a Wiener process, and
    \begin{gather*}
        \eta(t) = \int_0^t \nu(W_s)\,ds,
        \quad
        \nu(x)= \frac{1}{a_+} \I_{(0,\infty)}(x) + \frac{1}{a_-} \I_{(-\infty,0)}(x).
    \end{gather*}
\end{theorem}
\begin{proof}
Let us check the conditions \ref{th-cond-1i}), \ref{th-cond-1ii}) of Theorem \ref{th-cond-1}. Under our assumptions, we have for any $n \geq 1$ \begin{gather*}
    Z_n (t)
    =\frac{B_{nt}}{\sqrt{n}}
    = W_n (t),
\end{gather*}
for some Wiener process $W_n$.

\medskip\noindent
Condition \ref{th-cond-1i}): Note that the processes $Z_n$ and $\frac 1n\tau_{n\,\cdot}$ are continuous. Furthermore, for any $k \geq 1$, any $0 \leq t_1 < t_2 <\ldots< t_k \leq T$ and any $x_i, z_i \in \real$, $1 \leq i \leq k$, we have
\begin{align*}
    \Pp &\left\lbrace \frac{\tau_{n t_i}}{n} \leq x_i, \; W_n (t_i) \leq z_i, \; \forall 1\leq i \leq k  \right\rbrace \\
    &=\Pp \left\lbrace \int_0^{nx_i} \frac{ds}{\lambda(B_s)} \geq nt_i, \; W_n (t_i) \leq z_i, \; \forall 1\leq i \leq k  \right\rbrace \\
    &=\Pp \left\lbrace \int_0^{x_i} \frac{ds}{\lambda(B_{ns})} \geq t_i, \; W_n (t_i) \leq z_i, \; \forall 1\leq i \leq k  \right\rbrace \\
    &= \Pp \left\lbrace \int_0^{x_i} \frac{ds}{\lambda(\sqrt{n} W_n (s))} \geq t_i, \; W_n (t_i) \leq z_i, \; \forall 1\leq i \leq k  \right\rbrace \\
    &= \Pp \left\lbrace \int_0^{x_i} \frac{ds}{\lambda(\sqrt{n} W (s))} \geq t_i, \; W (t_i) \leq z_i, \; \forall 1\leq i \leq k  \right\rbrace,
\intertext{where $W$ is a standard Wiener process. Further,}
    \Pp&\left\lbrace \int_0^{x_i} \frac{ds}{\lambda(\sqrt{n} W (s))} \geq t_i, \; W (t_i) \leq z_i, \; \forall 1\leq i \leq k  \right\rbrace \\
    &= \Pp \left\lbrace \int_{\real} \Lcal^a_{x_i}(W)  \frac{da}{\lambda(\sqrt{n} a )}  \geq t_i, \; W (t_i) \leq z_i, \; \forall 1\leq i \leq k\right\rbrace.
\end{align*}
As before,  $\Lcal^a_{x}(W)$  is a local time of Wiener process at point $a$ in the time interval $[0,x]$. The assumptions $\lambda(x) \to a_\pm$ as $x \to \pm \infty$ and $\lambda(x) \geq c > 0$ enable us to apply Lebesgue's dominated convergence theorem. This gives, a.s.
\begin{align}\label{lim-dc}
    \lim_{n\to\infty} \int_{\real} \Lcal^a_{x_i}(W)  \frac{da}{\lambda(\sqrt{n} a )}
    = \frac 1{a_-} \int_{-\infty}^{0} \Lcal^a_{x_i}( W )\, da + \frac 1{a_+} \int_{0}^{\infty} \Lcal^a_{x_i}(  W )\, da
    =: \eta(x_i).
\end{align}
Then, obviously,
\begin{align*}
    \Pp &\left\lbrace \int_{\real} \Lcal^a_{x_i}(  W )  \frac{da}{\lambda(\sqrt{n} a )}  \geq t_i, \; W (t_i) \leq z_i, \;\forall 1\leq i \leq k\right\rbrace \\
    &\xrightarrow[n\to\infty]{} \Pp \big\lbrace \eta(x_i) \geq t_i, \; W (t_i) \leq z_i, \;\forall 1\leq i \leq k\big\rbrace
\end{align*}
for any $k\geq 1$, $t_i \in [0,T]$, $x_i \geq 0$ and $z_i \in \real$.

Note that the process $\eta(x)$ can be   represented as   $\eta(x) = \int_0^{x} \nu(W_s)\,ds$, where $\nu(y)= a_+^{-1} \I_{(0,\infty)}(y) + a_-^{-1} \I_{(-\infty,0)}(y)$. Evidently, $\eta(x+s)-\eta(x)>0$ for any $x>0$ and $s>0$ with probability $1$; moreover, $x\mapsto \eta(x)$ is a continuous function, and so the trajectories of $\eta$ are strictly increasing in $x>0$ with probability $1$. Thus, the inverse process $\zeta(t):= \eta^{-1} (t)$ exists, and is a continuous process. In particular, we see that both limit processes are continuous. Finally,
\begin{align*}
    \Pp&\left\lbrace \frac{\tau_{n t_i}}{n} \leq x_i, \; W_n (t_i) \leq z_i, \;\forall 1\leq i \leq k  \right\rbrace\\
    &\xrightarrow[n\to\infty]{} \Pp \big\lbrace \zeta(t_i) \leq x_i , \; W (t_i) \leq z_i, \;\forall 1\leq i \leq k  \big\rbrace,
\end{align*}
and the first condition is established.

\medskip\noindent
Condition \ref{th-cond-1ii}): Obviously, we have for every $\epsilon>0$
\begin{align*}
    \lim_{h \downarrow 0} \limsup\limits_{n\to \infty} & \,\Pp \left\lbrace \sup \limits_{0 \leq t_1 \leq t_2 \leq (t_1+h) \wedge T} |S_n (t_2) - S_n (t_1)|> \epsilon\right\rbrace \\
    &=\lim_{h \downarrow 0}  \Pp \left\lbrace \sup \limits_{0 \leq t_1 \leq t_2 \leq (t_1+h)  \wedge T} |W(t_2) - W (t_1)|> \epsilon \right\rbrace=0.
\end{align*}
This gives condition \ref{th-cond-1ii}), and an application of Theorem~\ref{th-cond-1} finishes the proof.
\end{proof}

\begin{remark} If the limits   $a_\pm$  of $\lambda(x)$ at $\pm \infty$ coincide,   we are in   a situation  which resembles the central limit theorem and the invariance principle. If, however, the limits are different, we get, in the limit, a process which is equivalent in measure to a skew Brownian motion.
\end{remark}

Our next aim is to avoid the condition that the intensity   $\lambda$ is separated from zero.   From now on, we will only assume that $\lambda$ is   (Lebesgue a.e.)  strictly positive with strictly positive limits as $x\to\pm\infty$ and   that $1/\lambda$  is locally integrable.

\begin{theorem}\label{lim-thm-2}
    Let $\lambda:\real\to (0,\infty)$ be measurable such that $1/\lambda$ is locally integrable and there exist two numbers, $a_\pm >0$ such that
    \begin{gather*}
        \lim_{x\to\pm\infty}\lambda(x) = a_\pm.
    \end{gather*}
    Then, on any interval $[0,T]$
    \begin{gather*}
        \left\{\frac{B_{\tau_{nt}}}{\sqrt{n}},\,t\in[0,T]\right\}
        \xRightarrow[]{\phantom{\mathrm{d}}}
        \left\{W(\eta^{-1} (t)),\, t\in[0,T]\right\},
    \end{gather*}
    where $W$ is a Wiener process, and $\eta$, $\nu$ are as in Theorem~\ref{lim-thm-1}.
\end{theorem}
\begin{proof}
The only place in the proof of Theorem \ref{lim-thm-1} where the condition $\lambda(x) \geq c > 0$, $x \in \real$, is used, is the limit \eqref{lim-dc} which we calculated by dominated convergence. We will now show how we can avoid this theorem.

It is enough to consider the integral $\int_{0}^\infty \Lcal^a_{x_i}(  W ) \lambda(\sqrt{n} a )^{-1}\,da$, since we can handle the integral $\int_{-\infty}^0 \Lcal^a_{x_i}(  W ) \lambda(\sqrt{n} a )^{-1}\,da$ in exactly the same way. Fix $0<\beta<\frac 12$; for any $\epsilon>0$ there exist some $n_0=n_0(\epsilon)$ such that
\begin{gather*}
    \left|\frac{1}{\lambda(x)}- \frac 1{a_+} \right|<\epsilon
    \quad\text{for every\ } x>n^{1/2-\beta}_0.
\end{gather*}
Using $\int_{\real}\Lcal^a_{x}(B)da=x$ for every $x>0$, we see that for any $n>n_0$
\begin{align}
&\notag\left|\int_{0}^\infty \Lcal^a_{x_i}(W) \frac{da}{\lambda(a\sqrt{n})} - \frac 1{a_+} \int_{0} ^{+ \infty} \Lcal^a_{x_i}(W)\,da\right|\\
&\notag\leq  \int_{0}^{n^{-\beta}} \Lcal^a_{x_i}(W)\frac{da}{\lambda(a\sqrt{n})}
    +\frac 1{a_+} \int_{0}^{n^{-\beta}} \Lcal^a_{x_i}(W)\,da
    +\int_{n^{-\beta}}^{\infty} \Lcal^a_{x_i}(W) \left|\frac{1}{\lambda(a\sqrt{n})}-\frac 1{a_+} \right| da\\
&\notag\leq \sup_{0\leq a\leq n^{-\beta}} \Lcal^a_{x_i}(W) \frac 1{\sqrt n}\int_0^{n^{1/2-\beta}}\frac{da}{\lambda(a)}
    +\frac 1{a_+} \int_{0}^{n^{-\beta}} \Lcal^a_{x_i}(W)\,da
    +\epsilon \int_{n^{-\beta}}^{\infty} \Lcal^a_{x_i}(W)\,da\\
&\label{ineq-for-int}\leq \sup_{0\leq a\leq 1}\Lcal^a_{x_i}(W) \frac 1{\sqrt n}\int_0^{n^{1/2-\beta}}\frac{da}{\lambda(a)}
    + \frac 1{a_+} \int_{0}^{n^{-\beta}} \Lcal^a_{x_i}(W)\,da
    + \epsilon x_i.
\end{align}

We will now let $n\to \infty$ in \eqref{ineq-for-int}. It is well known, see e.g.\ \cite{rev-yor}, that local time $\Lcal^a_{x}(W)$ of the Wiener process has a modification which is locally H\"{o}lder continuous in $a$ up to order $1/2$. Therefore, $\sup_{0\leq a\leq 1}\Lcal^a_{x_i}(W)$ is bounded a.s., and
\begin{gather}\label{auxiliary}
    \lim_{n\to\infty} \frac 1{\sqrt n}\int_0^{n^{1/2-\beta}}\frac{da}{\lambda(a)}
    \leq
    \lim_{n\to\infty}\left(\frac 1{\sqrt n}\int_0^{n_0^{1/2-\beta}}\frac{da}{\lambda(a )} + \left(\frac{1}{a_+}+\epsilon\right) \frac 1{n^{\beta}}\right)
    =0.
\end{gather}
Since $\epsilon$ is arbitrary, \eqref{ineq-for-int} and \eqref{auxiliary} show that $\lim\limits_{n\to\infty}\int_{0}^\infty \Lcal^a_{x_i}(W) \frac{da}{\lambda(a\sqrt{n})} = \frac 1{a_+} \int_{0} ^{+ \infty} \Lcal^a_{x_i}(W)\,da$, and we get \eqref{lim-dc}.
\end{proof}

\begin{remark}\label{rem-gen}
    Let us calculate the infinitesimal generator of the Markov process $Y^{(n)} = \left\{\frac{1}{\sqrt{n}}B_{\tau_{nt}},\,t\geq 0\right\}$, using the martingale representation \eqref{mart-repres}. The sequence of processes
    \begin{gather*}
        Y^{(n)}_t = \frac{B_{\tau_{nt}}}{\sqrt{n}} = \frac{Y_{nt}}{\sqrt{n}}
    \end{gather*}
    satisfies the equation
    \begin{align*}
        Y^{(n)}_t
        = \frac 1{\sqrt n} \int_0^{nt}\lambda^{1/2}(Y_s)\,d\tilde{B}_s
        &= \frac 1{\sqrt n} \int_0^t\lambda^{1/2}(Y_{nz})\,d\tilde{B}_{nz}\\
        &= \int_0^t\lambda^{1/2}\left(\sqrt nY^{(n)}_{z}\right)\,d\tilde{B}^{(n)}_{z},
    \end{align*}
    where $\tilde{B}^{(n)}_{z}=\frac{1}{\sqrt{n}}\tilde{B}_{nz}$ is a Wiener process. Therefore, the generator of $Y^{(n)}$ is
    $A^{(n)}f(x)= \lambda(x\sqrt n) f''(x)$;    if $\lambda$ has limits as $x\to\pm\infty$,   weak convergence of the processes $Y^{(n)}$ corresponds to the pointwise convergence of the generators. Let us describe the situation more precisely. Observe that
    \begin{gather*}
        \lambda^{(n)}(x) = \lambda(x\sqrt n)
        \xrightarrow[n\to\infty]{}  a_+\I_{(0,\infty)}(x) + a_-\I_{(-\infty,0)}(x) + a_0 \I_{\{0\}}(x).
    \end{gather*}
    The limit process satisfies the following stochastic differential equation
    \begin{gather}\label{limit-gen}
        Y_t
        = \int_0^t\left(\sqrt{a_+} \I_{(0,\infty)}(Y_s) + \sqrt{a_-}\I_{(-\infty,0)}(Y_s) + \sqrt{a_0} \I_{\{0\}}(Y_s)\right)d\tilde{B}_s,
    \end{gather}
    with some Wiener process $\tilde{B}$. If the coefficients $a_-,a_+,a_0$ are strictly positive, then the diffusion coefficient of this equation,
    \begin{gather}
        \sigma(x)=\sqrt{a_+}\I_{(0,\infty)}(x)+ \sqrt{a_-}\I_{(-\infty,0)}(x)+\sqrt{a_0}\I_{\{0\}}(x),
    \end{gather}
    is bounded, measurable   and uniformly bounded away from $0$.   According to the well-known  Krylov theorem (see, e.g.\ Proposition~1.15 in \cite{cher}), the equation \eqref{limit-gen} has a weak solution which is unique in law.

    Pick such a Wiener process $\tilde{B}$,   and consider the (pathwise)  solution of the SDE \eqref{limit-gen} driven by $\tilde B$.   Then for any $t>0$
    \begin{gather*}
        \int_0^t\I_{\{0\}}(Y_s)\,ds
        = \int_0^t\frac{\I_{\{0\}}(Y_s)}{\sigma^2(Y_s)}\,d\langle Y\rangle_s
        = \int_{\real} \frac{\I_{\{0\}}(x)}{\sigma^2(x)} \Lcal^x_t(Y)\,dx
        = 0
        \quad\text{a.s.}
    \end{gather*}
    as before, $\Lcal^x_t(Y)$ dentes the local time of $Y$ at the point $x$ and up to time $t$. This shows that   the law of $Y$ coincides with the law of any weak solution $Z$ of the following SDE driven by a standard Wiener process $W$
    \begin{gather*}
        Z_t
        = \int_0^t\left(\sqrt{a_+} \I_{(0,\infty)}(Z_s) + \sqrt{a_-} \I_{(-\infty,0)}(Z_s)\right)   dW_s.
    \end{gather*}

    Without loss of generality, we may change $\lambda$ at $x=0$ and assume that $\lambda(0)  =  0$,  because it will not change the value of $S_B(t)$ and all further calculations. Indeed, define $\frac{0}{0}:=0$. With this convention and by Fatou's lemma, we find for any $t>0$ and any standard Wiener process $W$
    \begin{gather*}
        \int_0^t \frac{\I_{\{0\}}(W_s)}{\lambda (W_s)}\, ds
        = \int_0^t\liminf_{\epsilon\downarrow 0}\frac{\I_{\{0\}}(W_s)}{\epsilon+\lambda (W_s)} \,ds
        \leq \liminf_{\epsilon\downarrow 0}\int_{\real} \Lcal^x_t(W)\frac{\I_{\{0\}}(x)}{\epsilon+\lambda (x)}\, dx
        = 0.
    \end{gather*}
\end{remark}

We will see in  Section~\ref{sub3} that the   convergence of the normalizing time change process does not require pointwise convergence of $\lambda(x)$ as $x\to\pm\infty$ but only some kind of mean convergence, cf.~\ref{theor-gen1ii}) in Theorem~\ref{theor-gen1}. The corresponding processes still converge in measure.

\subsection{Functional limit theorems in the case where the intensity  converges in average}\label{sub3}

We will now show that we can relax the conditions on $\lambda(\cdot)$ at infinity by an integral condition.   It is not hard to see that $\lim_{x\to\infty} \lambda(x)=a_+>0$ implies that $\lim_{x\to\infty} x^{-1}\int_0^x \lambda(s)^{-1}\,ds = a_+^{-1}$, while the converse is clearly not true. It turns out that an integral condition of this type is sufficient for the weak convergence of the time changed process.

\begin{theorem}\label{theor-gen1}
Assume that the following conditions are satisfied.
\begin{enumerate}[\upshape i)]
\item\label{theor-gen1i}
    Let $\lambda : \real \to [0,\infty)$ be a measurable function such that $\mathrm{Leb}\{x:\lambda(x)=0\}=0$ and $1/{\lambda}$ is    locally integrable.
\item\label{theor-gen1ii}
    There exist two numbers, $a_+, a_- >0$ such that for some $\delta>-1$ we have
    \begin{gather*}
        \lim_{x\to\infty}\frac{1}{x^{1+\delta}} \int_0^{x} \frac{ds}{\lambda(s)} =  \frac 1{a_+}
    \quad\text{and}\quad
        \lim_{x\to\infty}\frac{1}{x^{1+\delta}} \int_{-x}^0\frac{ds}{\lambda(s)} =  \frac 1{a_-}.
    \end{gather*}
\end{enumerate}
If \textup{\ref{theor-gen1i})--\ref{theor-gen1ii})} hold, then on any interval $[0,T]$
\begin{gather*}
    \left\{{n}^{-\frac{1}{2+\delta}}{B_{\tau_{nt}}}, \,t\in [0,T]\right\}
    \xRightarrow[]{\phantom{\mathrm{d}}}
    \left\{W(\eta_\delta^{-1} (t)),\, t\in[0,T]\right\},
\end{gather*}
where $W$ is a Wiener process, and
\begin{gather*}
    \eta_\delta(t)=\frac{2}{2+\delta}\int_0^t |W(s)|^{\delta} \nu(W(s))\, ds
     \quad\text{with}\quad
    \nu(w) = \frac 1{a+}\I_{(0,\infty)}(w)+\frac 1{a-}\I_{(-\infty,0)}(w).
\end{gather*}
\end{theorem}
\begin{proof}
Pick some $\alpha\in(0,1)$ and consider the process
\begin{gather*}
    \frac{B_{\tau_{nt}}}{ n^\alpha}=\frac{B_{n^{2\alpha}\frac{\tau_{nt}}{n^{2\alpha}}}}{ n^\alpha}.
\end{gather*}
In this case, the pre-limit Wiener processes have the form $\tilde{W}_n(s) = n^{-\alpha} B_{n^{2\alpha}s}$, and we have to check weak convergence of the finite dimensional distributions of   the process $\left\{\left(\tilde{W}_n(t), \, n^{-2\alpha} \tau_{nt}\right),\, t\in [0,T]\right\}$ or of $\left\{\left(\tilde{W}_n(t),\, \frac 1n\int_0^{n^{2\alpha}t} \lambda(B_s)^{-1}\,ds\right),\, t\in [0,T]\right\}$, $T>0$.

Consider the function $F(x)= \int_0^x \int_0^y \lambda(s)^{-1}\,ds\,dy$. Since $F'$ is continuous and $F''\left(x\right)$ exists Lebesgue a.e.\ and is locally integrable, we can apply the generalized It\^{o} formula to $F(B_t)$, see Theorem \ref{ITO}. This gives
\begin{gather*}
    F(B_t)=\int_0^t \int_0^{B_s} \frac{du}{\lambda(u)}\,dB_s + \frac12 \int_0^t \frac{ds}{\lambda(B_s)}.
\end{gather*}
Therefore,
\begin{gather}\label{equv-f}
    \int_0^t \frac{ds}{\lambda(B_s)}= 2F(B_t)-2 \int_0^t \int_0^{B_s} \frac{du}{\lambda(u)}\,dB_s.
\end{gather}
From \eqref{equv-f} we conclude that
\begin{align*}
    \frac 1n \int_0^{n^{2\alpha} x_i} \frac{ds}{\lambda(B_s)}
    &= \frac2n \int_0^{n^\alpha \tilde{W}_{n} (x_i)}  \int_0^y \frac{ds}{\lambda(s)}\, dy
       - \frac{2}{n^{1-\alpha}} \int_0^{x_i} \int_0^{n^\alpha \tilde{W}_n({s})} \frac{du}{\lambda(u)}\,d\tilde{W}_n({s})\\
    &= \frac{2}{n^{1-\alpha}}  \int_0^{  \tilde{W}_{n} (x_i)}  \int_0^{n^\alpha y} \frac{ds}{\lambda(s)} \, dy
       - \frac{2}{n^{1-\alpha}}  \int_0^{x_i} \int_0^{ {n}^\alpha \tilde{W}_{n}(s)}\frac{du}{\lambda(u)}\,d\tilde{W}_n(s).
\end{align*}
Consequently,
\begin{align*}
    &\Pp\left\lbrace \frac{\tau_{n t_i}}{n^{2\alpha}} \leq x_i, \; \tilde{W}_n (t_i) \leq z_i, \; \forall 1\leq i \leq k  \right\rbrace \\
    &=\Pp \left\lbrace \frac1n \int_0^{n^{2\alpha} x_i} \frac{ds}{\lambda(B_s )} \geq   t_i, \; \tilde{W}_n (t_i) \leq z_i, \;\forall 1\leq i \leq k \right\rbrace \\
    &=\Pp \left\{ \frac{2}{n^{1-\alpha}}\left[\int_0^{ {W} (x_i)}\int_0^{n^\alpha y} \frac{ds}{\lambda(s)}\,dy
       - \int_0^{x_i} \int_0^{n^\alpha W(s)} \frac{du}{\lambda(u)}\,dW(s)\right]\geq   t_i, \; W (t_i) \leq z_i,\right.\\
    &\left.\phantom{=\Pp \left\{ \frac{2}{n^{1-\alpha}}\left[\int_0^{ {W} (x_i)}\int_0^{n^\alpha y} \frac{ds}{\lambda(s)}\,dy
       - \int_0^{x_i} \int_0^{n^\alpha W(s)} \frac{du}{\lambda(u)}\,dW(s)\right]\geq   t_i, \;\right.}   \;\forall 1\leq i \leq k  \right\}.
\end{align*}
From the conditions \ref{theor-gen1i}) and \ref{theor-gen1ii}) we conclude that $\alpha=1/(2+\delta)$, and that
\begin{align*}
    \frac{2}{n^{1-\alpha}} \int_0^{W(x)} \int_0^{n^\alpha y} \frac{ds}{\lambda(s)} \, dy
    \xrightarrow[n\to\infty]{}
     \frac 2{a_+} \int_0^{|W(x)|} y^{1+\delta}\,dy
    = \frac{2}{2+\delta}|W(x)|^{2+\delta}\nu(W(x))
\end{align*}
with probability $1$. Now we consider the continuous square-integrable martingale $M_n (x) := \int _0^x \phi _n (W_s) \, dW_s$, where
\begin{gather}\label{phi}
    \phi_n(u)
    = n^{-\frac{1+\delta}{2+\delta}} \int_0^{u{n^{1/(2+\delta)}}} \frac{dv}{\lambda(v)}
     \xrightarrow[n\to\infty]{} |u|^{1+\delta}\sgn(u)\nu(u).
\end{gather}
Because of the assumption~\ref{theor-gen1ii}), there exists some $x_0>0$ such that for all $|x|>x_0$
\begin{gather*}
    \frac{1}{|x|^{1+\delta}} \left|\int_0^{x} \frac{ds}{\lambda(s)}\right|
    \leq A = 2\left( \frac 1{a_+}+ \frac 1{a_-}\right).
\end{gather*}
Therefore, we can bound for all $n >x_0^{2+\delta}$ the function $|\phi _n (u)|$ in the following way:
\begin{align*}
    |\phi _n (u)|
    &\leq |\phi _n (1)|\I_{[0,1)}(|u|) + |u|^{1+\delta}\sup_{|x|\geq x_0}\frac{1}{|x|^{1+\delta}} \left|\int_0^{x} \frac{ds}{\lambda(s)}\right| \I_{(1,\infty)}(|u|)\\
    &\leq A\left(1+|u|^{1+\delta}\right).
\end{align*}
This shows that $|\phi _n (W_s)|\leq A\left(1+|W_s|^{1+\delta}\right)$. Using~\eqref{phi}, we see that
\begin{gather*}
    X_n (s) =\phi_n (W_s) \xrightarrow[n\to\infty]{} |W(s)|^{1+\delta}\sgn(W(s)) \nu(W(s))
\end{gather*}
with probability $1$. Thus, according to Theorem \ref{auxil-theo},
\begin{gather*}
    M_n (x)\xrightarrow[n\to\infty]{}
    \int_0^{x}|W(s)|^{1+\delta} \sgn(W(s)) \nu(W(s)) \, dW(s)
\end{gather*}
in probability. Finally, the following limit exists in probability:
\begin{align*}
    &\frac{2}{ {n}^{1-\alpha}} \int_0^{W(x)} \int_0^{ {n}^\alpha y} \frac{ds}{\lambda(s)} \, dy
    - \frac{2}{n^{1-\alpha}} \int_0^{x} \int_0^{n^\alpha  W(s)} \frac{du}{\lambda(u)}\, dW(s)\\
    &\xrightarrow[n\to\infty]{} \frac{2}{2+\delta}\,|W (x)|^{2+\delta} \nu(W(x))
    - 2\int_0^{x}|W(s)|^{1+\delta}\sgn(W(s)) \nu(W(s))\, dW(s).
\end{align*}

We can now apply the generalized It\^{o} formula (Theorem \ref{ITO}) separately to the functions $F_1(x)=x^{2+\delta}\I_{(0,\infty)}(x)$ and $\Phi_2(x)=|x|^{2+\delta}\I_{(-\infty,0)}(x)$ -- both functions have continuous first derivatives and Lebesgue a.e.\ existing,  locally integrable second derivatives. This finishes the proof.
\end{proof}

For $\delta=0$ we get, in particular, the following result.
\begin{corollary}\label{part-case}
Assume that the following conditions are satisfied.
\begin{enumerate}[\upshape i)]
\item\label{part-casei}
    Let $\lambda : \real \to [0,\infty)$ be a measurable function such that $\mathrm{Leb}\{x:\lambda(x)=0\}=0$ and $1/{\lambda}$ is    locally integrable.
\item\label{part-caseii}
    There exist two numbers, $a_+, a_- >0$ such that
    \begin{gather*}
        \lim_{x\to\infty}\frac{1}{x} \int_0^{x} \frac{ds}{\lambda(s)} =  \frac 1{a_+}
    \quad\text{and}\quad
        \lim_{x\to\infty}\frac{1}{x} \int_{-x}^0\frac{ds}{\lambda(s)} =  \frac 1{a_-}.
    \end{gather*}
\end{enumerate}
If \textup{\ref{theor-gen1i})--\ref{theor-gen1ii})} hold, then on any interval $[0,T]$
\begin{gather*}
    \left\{\frac{B_{\tau_{nt}}}{\sqrt{n}},\;t\in[0,T]\right\}
    \xRightarrow[]{\phantom{\mathrm{d}}}
    \left\{W(\eta ^{-1} (t)),\, t\in[0,T]\right\}
\end{gather*}
where $W$ is a Wiener process, and $\eta $ is as in Theorem \ref{lim-thm-1}.
\end{corollary}

Theorem \ref{theor-gen1} can be further generalized to the case of normalizing factors which are not necessarily power functions, but in some sense, close to power ones.   Recall that a measurable function $f:\real\to(0,\infty)$ is \emph{regularly varying} (\emph{at infinity}) of order $\rho\in\real$, if it is of the form $f(z)=z^\rho L(z)$; the function $L(z)$ is positive, measurable and \emph{slowly varying at infinity}, i.e.\ $\lim_{z\to\infty} L(tz)/L(z)=1$ exists for all $t>0$. Typical examples of slowly varying functions are $\log z$ or $\log^\alpha |\log z|$.
 Our standard reference for regularly varying functions is \cite{bgt}.

Since the proof of the next theorem is similar to the proof of Theorem \ref{theor-gen1}, it is omitted.
\begin{theorem}\label{thm-rv}
Assume that the following conditions are satisfied.
\begin{enumerate}[\upshape i)]
\item\label{thm-rv-i}
    Let $\lambda : \real \to [0,\infty)$ be a measurable function such that $\mathrm{Leb}\{x:\lambda(x)=0\}=0$ and $1/{\lambda}$ is    locally integrable.
\item\label{thm-rv-ii}
    There is a strictly increasing positive function $\psi:\real\to [0,\infty)$ whose inverse function $\psi^{-1}$ is regularly varying of index $\gamma>1$ at infinity\footnote{This is equivalent to saying that $\psi$ is regularly varying at infinity of order $1/\gamma$.} such that the following limits exist
    \begin{gather*}
        \lim_{x\to\infty} \frac{x}{ \psi^{-1}(x)} \int_0^{x} \frac{ds}{\lambda(s)} =  \frac 1{a_+}
        \quad\text{and}\quad
        \lim_{x\to\infty} \frac{x}{ \psi^{-1}(x)}  \int_{-x}^0 \frac{ds}{\lambda(s)}
        =  \frac 1{a_-}.
    \end{gather*}
\end{enumerate}
If \textup{\ref{thm-rv-i})--\ref{thm-rv-ii})} hold, then on any interval $[0,T]$
\begin{gather*}
    \left\{\frac{B_{\tau_{nt}}}{\psi(n)},\,t\in[0,T]\right\}
    \xRightarrow[]{\phantom{\mathrm{d}}}
    \left\{W(\eta_\gamma^{-1} (t)),\, t\in[0,T]\right\},
\end{gather*}
where $W$ is a Wiener process,   $\nu(w) = \frac 1{a_+}\I_{(0,\infty)}(w) + \frac 1{a_-}\I_{(-\infty,0)}(w)$,   and $\eta_\gamma(t)=\frac{2}{\gamma}\int_0^t |W(s)|^{\gamma-2} \nu(W(s)) \,ds$.
\end{theorem}

\begin{remark}
    Clearly, Theorem \ref{thm-rv} covers the case where function $\lambda$ is positive and periodic. Without loss of generality, let $\lambda$ be periodic with period $1$. Then
    \begin{gather*}
        \frac1x \int_0^x \frac{ds}{\lambda(s)}
        \sim \frac 1n \int_0^n \frac{ds}{\lambda(s)}
        = \int_0^1 \frac{ds}{\lambda(s)},
    \end{gather*}
    as $n\to \infty$ and $x\in[n,n+1)$.  In particular, $a_+ = a_- = \left(\int_0^1 \lambda(s)^{-1}\,ds\right)^{-1}$.
\end{remark}

\section{Appendix}

The following result is Theorem~4.6 from \cite{kurtz} which we adapt for our purposes.
\begin{theorem}\label{auxil-theo}
    Let $W=\{W(t),\,t\geq 0\}$ be a Wiener process, $\Fcal^W$ its canonical filtration and $X_n=\{X_n(t),\, t\ge 0\}$, $n\geq 1$, a sequence of continuous processes which are adapted to the filtration $ \Fcal^{W}$. Assume that, for some constants $C>0$, $\beta>0$, the following estimate holds
    \begin{gather*}
        |X_n (t)|\leq C|W (t)|^{\beta},\quad t>0.
    \end{gather*}
    If $X_n(t) \xrightarrow[n\to\infty]{} X(t)$ in probability for any $t>0$, then
    \begin{gather*}
        \left(X_n, W, {\textstyle\int_0^\cdot} X_n\, dW\right)  \xrightarrow[n\to\infty]{} \left (X, W, {\textstyle\int_0^\cdot} X dW\right)
    \end{gather*}
    in probability.
\end{theorem}

The next result is the generalized It\^{o} formula, see \cite{krylov}, Theorem 4, adapted for our purposes.
\begin{theorem}[generalized It\^o formula]\label{ITO}
    Let $W$ be a Wiener process. Assume that the function $F:\real\to\real$ is continuously differentiable and that the second derivative $F''$ exists Lebesgue a.e.\ and is  locally integrable. Then the following identity holds for all $t>0$ with probability $1$:
    \begin{gather*}
        F(W(t))
        = F(0)+ \int_{0}^{t} F'(W(s))\, dW(s) + \frac{1}{2} \int_{0}^{t} F''(W(s))\, ds.
    \end{gather*}
\end{theorem}

\bibliographystyle{agsm}
\bibliography{kms}

\end{document}